\documentclass[12pt]{article}
\usepackage{amsthm,color,amsmath,graphicx,mathpazo}
\usepackage{graphicx,anysize}
\usepackage[colorlinks,urlcolor=blue,linkcolor=blue,citecolor=blue]{hyperref}
\marginsize{25mm}{25mm}{15mm}{15mm}

\newtheorem{theorem}{Theorem}[section]
\newtheorem{proposition}[theorem]{Proposition}
\newtheorem{lemma}[theorem]{Lemma}
\newtheorem*{them}{Theorem A}{\bf}{\it}

\theoremstyle{remark}
\newtheorem{remark}[theorem]{Remark}
\theoremstyle{definition}
\newtheorem{definition}[theorem]{Definition}

\allowdisplaybreaks
\numberwithin{equation}{section}

\renewcommand{\vec}[1]{\pmb{#1}}

\begin{document}

\title{Homology of saddle point reduction
and applications to resonant elliptic systems
\thanks{Supported by NSFC (11071237, 11171204) and RFDP (20094402110001).}
}

\author{Chong Li $^{\text{a}}$         \and
        Shibo Liu $^{\text{b}}$\thanks{email: liusb@xmu.edu.cn}
}

\date{\small\it $^{\text{a}}$ Institute of Mathematics,
Chinese Academy of Sciences,
Beijing 100190, P.R. China\\
$^{\text{b}}$ School of Mathematical Sciences, Xiamen University, Xiamen 361005, PR China}
\maketitle

\begin{abstract}
\noindent In the setting of saddle point reduction, we prove that the critical
groups of the original functional and the reduced functional are
isomorphic. As application, we obtain two nontrivial solutions for
elliptic gradient systems which may be resonant both at the origin and at
infinity. The difficulty that the variational functional does not
satisfy the Palais-Smale condition is overcame by taking advantage
of saddle point reduction. Our abstract results on critical groups
are crucial.\medskip

\smallskip

\noindent\emph{MSC20000: }58E05; 35J60
\end{abstract}

\section{Introduction}

Infinite dimensional Morse theory (see \cite{MR1196690,MR982267} for
a a systematic exploration) is very useful in obtaining multiple
solutions for nonlinear variational problems. The central concept in
this theory is the critical group $C_{\ast}( f,u) $ for a
$C^{1}$-functional $f:X\rightarrow\mathsf{R}$ at an isolated
critical point $u$. The critical group describes the local property
of $f$ near the critical point $u$. On the other hand, Bartsch and
Li \cite{MR1420790} introduced the critical group $C_{\ast}(
f,\infty) $ of $f$ at infinity, which describes the global property
of the functional $f$.

With these concepts we have the Morse inequalities%
\begin{equation}
\sum_{q=0}^{\infty}M_{q}t^{q}=\sum_{q=0}^{\infty}\beta_{q}t^{q}+(
1+t) Q( t)
\text{,} \label{e1}%
\end{equation}
where $Q$ is a formal series with nonnegative integer coefficients,%
\[
M_{q}=\sum\nolimits_{f^{\prime}( u) =0}\operatorname*{rank}C_{q}(
f,u)
\text{,\qquad}\beta_{q}=\operatorname*{rank}C_{q}( f,\infty) \text{.}%
\]
In most applications, we may distinguish critical points using
critical group, and we may find new critical points using the Morse
inequality. Therefore, the study of the critical group is very
important.

In 1979, Amann \cite{MR550724} established the theory of saddle
point reduction (also called Lyapu\-nov-Schmidt reduction in some
literature). Since then, saddle point reduction becomes an important
method in critical point theory, and has been widely applied to
various nonlinear boundary value problems
\cite{MR1302162,MR2651745,MR2685145,MR1662078,MR2488059,MR1781225,MR1090484}.

Let $( X,\left\langle \cdot,\cdot\right\rangle ) $ be a separable
Hilbert space with norm $\left\Vert \cdot\right\Vert $, and $f\in
C^{1}( X,\mathsf{R}) $. The basic assumption in saddle point
reduction is the following

\begin{description}
\item[$( A_{\pm}) $] $X^{\pm}$ are closed subspaces of $X$ such that
$X=X^{-}\oplus X^{+}$, and there exists a real number $\kappa>0$ such that%
\[
\pm\left\langle \nabla f( v+w_{1}) -\nabla f( v+w_{2}) ,w_{1}-w_{2}%
\right\rangle \geq\kappa\left\Vert w_{1}-w_{2}\right\Vert ^{2}%
\]
for all $v\in X^{-}$ and $w_{1},w_{2}\in X^{+}$.
\end{description}

\noindent Then by saddle point reduction, there exist $\psi\in C(
X^{-},X^{+}) $ and $\varphi\in C^{1}( X^{-},\mathsf{R}) $ such that
$\bar{v}$ is a critical point of $\varphi$ if and only if
$\bar{v}+\psi( \bar{v}) $ is a critial point of $f$; moreover
we have%
\begin{equation}
\varphi( v) := f( v+\psi( v) ) =\max\limits_{w\in X^{+}}f( v+w)\text{,}\qquad\forall v\in X^- \label{e}%
\end{equation}
for case $( A_{-}) $ and with `$\max$' replaced by `$\min$' for case
$( A_{+}) $, see \cite{MR679138} for a good proof of these results.
Thus, to find critical points of $f$ we may consider the reduced
functional $\varphi$. Since $\varphi$ is defined on a subspace, it
should be easier to study.

As mentioned before, Morse theory is a powerful tool in the study of
variational problems. If we want to apply Morse theory, naturally we
need to study the relation between the critical group of $\varphi$
and that of $f$. In our previous work \cite{MR2348521,MR2017717}, we
proved the results described in the following theorem.

\begin{them}
Let $X$ be a separable Hilbert space and $f\in C^{1}( X,\mathsf{R})
$.

\begin{description}
\item[\rm(i)] If $( A_{+}) $ holds, $f$ satisfies the Palais-Smale $( PS) $ condition
with critical values bounded from below, then $C_{q}( f,\infty)
\cong C_{q}( \varphi,\infty) $ for $q=0,1,2,\cdots$.

\item[\rm(ii)] If $( A_{-}) $ holds and $\mu=\dim X^{+}<\infty$, $f$ satisfies $( PS) $
condition with critical values bounded from below, then%
\begin{equation}
C_{q}( f,\infty) \cong C_{q-\mu}( \varphi,\infty) \text{,\qquad for
}q=0,1,2,\cdots\text{.} \label{e2}%
\end{equation}

\item[\rm(iii)] If $( A_{+}) $ holds and $\bar{v}\in X^{-}$ is an isolated critical point of
$\varphi$, then%
\[
C_{q}( f,\bar{v}+\psi( \bar{v}) ) \cong C_{q}( \varphi,\bar{v}) \text{,\qquad for }%
q=0,1,2,\cdots\text{.}%
\]

\end{description}
\end{them}

In view of Theorem A, we naturally expect that in case $( A_{-}) $
there should have a relation similar to \eqref{e2} for the critical
groups at isolated critical points. In this paper, we will establish
such a relation in the following theorem.

\begin{theorem}
\label{t1}Let $X$ be a separable Hilbert space and $f\in
C^{1}(X,\mathsf{R})$. Assume $(A_{-})$ holds and $\mu=\dim
X^{+}<\infty$; $\bar{v}\in X^{-}$ is
an isolated critical point of $\varphi$ such that $\varphi(\bar{v})$
is an isolated critical value. If moreover $\varphi$
satisfies the $( PS) $ condition, then%
\[
C_{q}(f,\bar{v}+\psi(\bar{v}))\cong
C_{q-\mu}(\varphi,\bar{v})\text{,\qquad
for }q=0,1,2,\cdots\text{.}%
\]
\end{theorem}

This theorem and Theorem A completely describe the relation between
the critical groups of the original functional $f$ and the reduced
functional $\varphi$.

Under the assumption of Theorem \ref{t1}, $\bar{v}+\psi(\bar{v})$ is an isolated critical point of $f$. It is also well known that if $f$ satisfies the $( PS) $ condition, so does $\varphi$, see \cite[Lemma 1]{MR1296115}. However, we emphasizes that in Theorem \ref{t1} it is not necessary to require that $f$ satisfies $( PS) $. This is very important in applications.

Let $\operatorname*{ind}( T,u) $ denotes the Leray--Schauder index
for $T:X\to X$ (a compact perturbation of the identity map) at its
isolated zero point $u$ and assume that $\nabla f$ is a compact
perturbation of the identity. Similar to \cite[Corollary
2.4]{MR2348521}, as a corollary of Theorem \ref{t1} and the
Poincar\'{e}-Hopf formula for $C^{1}$-functional \cite[Theorem
3.2]{MR2124871}
\[
\operatorname*{ind}( \nabla f,u) =\sum_{q=0}^{\infty}( -1) ^{q}%
\operatorname*{rank}C_{q}( f,u) \text{,}%
\]
if $( A_{-}) $ holds and $\mu=\dim X^{+}<\infty$, we have
\begin{equation}
\operatorname*{ind}( \nabla f,\bar{v}+\psi( \bar{v}) ) =( -1)
^{\mu}\operatorname*{ind}(
\nabla\varphi,\bar{v}) \text{.} \label{e3}%
\end{equation}
The corresponding result for case $( A_{+}) $, namely
\cite[Corollary 2.4]{MR2348521}, is originally due to Lazer and
McKenna \cite{MR787722}. As far as we know, the identity \eqref{e3}
does not appear elsewhere.

Our investigation of the relation between the critical groups of the
original functional and the reduced functional is motivated by the
study of multiple solutions for nonlinear boundary value problems.
In the second part of this paper, as applications of our abstract
results we consider asymptotically linear elliptic systems. Such
problems have attracted some attentions in recent years, see
\cite{MR2651745,MR2685145,MR2227915,MR1813819,MR1697052}.

Our assumptions on the nonlinearity are so weak that the
corresponding Euler-Lagrange functional does not satisfy the $( PS)
$ condition. Nevertheless, using some idea from
\cite{MR2488059,MR1781225}, by taking advantage of saddle point
reduction we can overcome this difficulty. As we will see in Remark
\ref{rek2}, because the asymptotic limits may be different variable
matrices, the local linking argument used in
\cite{MR2488059,MR1781225} does not apply. To prove our multiplicity
results (Theorems \ref{st1} and \ref{st2}), Theorem A (iii) and
Theorem \ref{t1} are crucial.

\section{Proof of Theorem \ref{t1}}

Let $X$ be a Banach space and $f\in C^{1}( X,\mathsf{R}) $. Let $u$
be an
isolated critical point of $f$ with critical
value $c=f( u) $, $\Omega$ be an arbitrary neighborhood of $u$. Then the group%
\[
C_{q}( f,u) =H_{q}( f_{c}\cap\Omega,(f_{c}\cap\Omega)\backslash\left\{  u\right\}  )
\text{,\qquad }q=0,1,2,\cdots
\]
is called the $q^{\operatorname*{th}}$ critical group of $f$ at $u$.
Here $f_{c}=f^{-1}(-\infty,c]$, $H_{q}( A,B) $ stands for the
$q^{\operatorname*{th}}$ singular relative homology group of the
topological pair $( A,B) $ with coefficients in a field
$\mathcal{G}$. By the excision property of homology, the critical
groups of $f$ at $u$ described the local property of $f$ near $u$.

If $f$ satisfies the $( PS) $ condition and the critical values of
$f$ are bounded from below by $\alpha\in\mathsf{R}$, then according
to
\cite[Definition 3.4]{MR1420790}, the group%
\begin{equation}
C_{q}( f,\infty) =H_{q}( X,f_{\alpha}) \text{,\qquad}q=0,1,2,\cdots\label{e4}%
\end{equation}
is called the $q^{\operatorname*{th}}$ critical group of $f$ at
infinity. Since $f$ satisfies $( PS) $, by the deformation lemma,
the right hand side of \eqref{e4} does not depend on the choose of
$\alpha$. Since all critical points of $f$ are contained in
$X\backslash f_{\alpha}$, we can say that the critical groups of $f$
at infinity describe the global property of $f$.

From the definitions of critical groups, we see that analytically,
$C_{\ast}( f,u) $ is simpler than $C_{\ast}( f,\infty) $, because the former does not require the $(PS)$ condition; while
topologically, $C_{\ast}( f,u) $ is more complicated than $C_{\ast}(
f,\infty) $, because the topological pair on the right hand side of \eqref{e4} is simpler. This explains why the results in Theorem A (i) and (ii)
were proved first.

\begin{proof}
[Proof of Theorem \ref{t1}]
Assume $\varphi(\bar{v})=f(\bar{v}+\psi(\bar{v}))=a$. Since $a$ is an isolated
critical value of $\varphi$, there is an $\varepsilon>0$ such that $\varphi$
has no critical value in $(a,a+\varepsilon]$. Since $\varphi$ satisfies $(PS)$, by the second deformation lemma
\cite{MR1196690,MR891261}, there is a continuous $\eta:\left[  0,1\right]
\times\varphi_{a+\varepsilon}\rightarrow\varphi_{a+\varepsilon}$ such that%
\begin{equation}
\left.
\begin{array}
[c]{ll}%
\eta(0,u)=u\text{,} & u\in\varphi_{a+\varepsilon}\text{,}\\
\eta(1,\varphi_{a+\varepsilon})\subset\varphi_{a}\text{,} & \\
\eta(t,u)=u\text{,} & (t,u)\in\left[  0,1\right]  \times\varphi_{a}\text{.}%
\end{array}
\right\}  \label{defo}%
\end{equation}
Let $O\subset\varphi_{a+\varepsilon}$ be a neighborhood of $\bar{v}$ such that
$\varphi$ has no critical point in $O\backslash\left\{  \bar{v}\right\}  $,
and set%
\[
U=\bigg(  \bigcup_{t\in\left[  0,1\right]  }\eta(t,O)\bigg)  \cup\varphi
_{a}\text{.}%
\]
Then $U$ is an $\eta$-invariant neighborhood of $\bar{v}$, and $\Omega=U\times X^{+}$ is a
neighborhood of $(\bar{v},\psi(\bar{v}))$.

By the property of $\varphi$ and $\psi$ described in \eqref{e}, if
$\varphi(v)\leq a$, then for any $w\in X^{+}$ we have $f(v+w)\leq a$. Thus,
setting%
\[
\Theta=\left\{  \left.  (v,w)\right\vert \,f(v+w)\leq a,\varphi(v)>a\right\}
\text{,}%
\]
we have $f_{a}=(\varphi_{a}\times X^{+})\cup\Theta$.

Under the assumption $(A_{-})$, it has been shown in the proof of
\cite[Theorem 1.2]{MR2017717} that $f_{a}$ is homotopically equivalent to%
\[
A=(\varphi_{a}\times X^{+})\cup\left\{  \left.  (v,w)\right\vert
\,\varphi(v)>a,w\neq\psi(v)\right\}
\]
via a homotopy $F:\left[  0,1\right]  \times A\rightarrow A$ constructed in
that proof (using condition $(A_{-})$ and the implicit function theorem).
Moreover, denoting%
\[
S=\left\{  \left.  w\in X^{+}\right\vert \,w\neq0\right\}  \text{,}%
\]
a homeomorphism $G$ between $A$ and%
\[
B=(\varphi_{a}\times X^{+})\cup\big((X^{-}\backslash\varphi_{a})\times S\big)
\]
has also been given there. The deformations $F$ and $G$ have been illustrated
in Figure \ref{fig1}, where the thick segments with endpoint $\bar{v}$ represent the level set $\varphi_a$; while the shadowed rigions in the three subfigures
represent the sets $f_{a}$, $A$ and $B$ respectively.

\begin{figure}[h]
\centerline{\includegraphics[width=0.99
\textwidth]{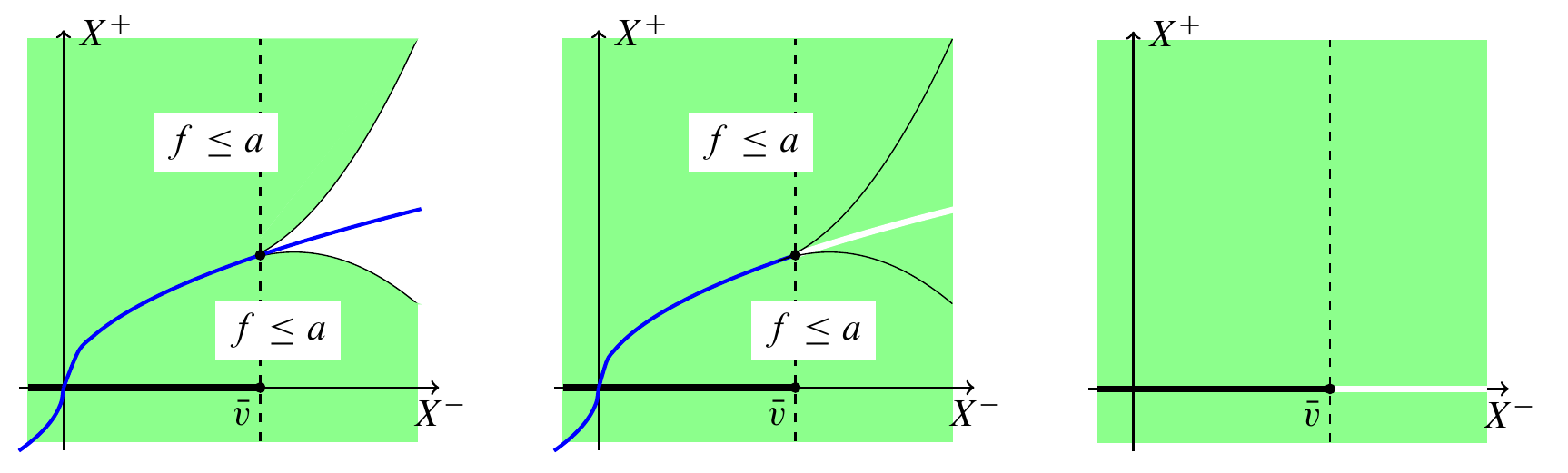}}\caption{Deformations of the level set $f_a$.}
\label{fig1}
\end{figure}

Let $\tilde{F}$ be the homotopy inverse of $F$, then $\Gamma=G\circ\tilde{F}$
is a homotopic equivalance between $f_{a}$ and $B$. By the definitions of $F$
and $G$ (see the proof of \cite[Theorem 1.2]{MR2017717}), we see that $\Gamma$
does not change the $v$-variable. Therefore, restricting $\Gamma$ to
$f_{a}\cap\Omega$, we obtain%
\[
f_{a}\cap\Omega\approx B\cap\Omega=( \varphi_{a}\times X^{+}) \cup\big( \big(
(X^{-}\backslash\varphi_{a})\cap U\big) \times S\big)
\]

Since $U$ is invariant under the flow $\eta$,  we can define $H:\left[  0,1\right]  \times( B\cap\Omega) \rightarrow
B\cap\Omega$,%
\[
H(t,(v,w))=\left\{
\begin{array}
[c]{ll}%
(v,w)\text{,} & \text{if }(v,w)\in\varphi_{a}\times X^{+}\text{,}\\
( \eta( t,v) ,w) \text{,} & \text{if }(v,w)\in\big( (X^{-}\backslash\varphi
_{a})\cap U\big) \times S\text{.}%
\end{array}
\right.
\]
Using \eqref{defo}, it is easy to see that $H$ is continuous, and $H(1,\cdot)$
is a homotopic equivalance between $B\cap\Omega$ and $\varphi_{a}\times X^{+}%
$. Combining the above homotopies, we can deform $f_{a}\cap\Omega$ to
$\varphi_{a}\times X^{+}$ continuously. The deformation maps $(\bar{v}%
,\psi(\bar{v}))$ to $(\bar{v},0)$, therefore it is also a homotopic
equivalance between $( f_{a}\cap\Omega) \backslash(\bar{v},\psi(\bar{v}))$ and
$(\varphi_{a}\times X^{+})\backslash(\bar{v},0)$.

Noting that $(\varphi_{a}\times X^{+})\backslash(\bar{v},0)=(\varphi_{a}\times
S)\cup((\varphi_{a}\backslash\bar{v})\times X^{+})$, we have%
\begin{align*}
\big( f_{a}\cap\Omega,( f_{a}\cap\Omega) \backslash(\bar{v},\psi(\bar{v}))\big)  &
\simeq( \varphi_{a}\times X^{+},(\varphi_{a}\times X^{+})\backslash(\bar
{v},0))\\
&  =\big( \varphi_{a}\times X^{+},(\varphi_{a}\times S)\cup((\varphi_{a}\backslash
\bar{v})\times X^{+})\big)\\
&  =(\varphi_{a},\varphi_{a}\backslash\bar{v})\times(X^{+},S)\text{.}%
\end{align*}
Passing to homology and applying the K\"{u}nneth formula, we deduce%
\begin{align*}
C_{\ast}(f,\bar{v}+\psi(\bar{v}))  &  =H_{\ast}\big( f_{a}\cap\Omega,( f_{a}%
\cap\Omega) \backslash(\bar{v},\psi(\bar{v}))\big)\\
&  \cong H_{\ast}( (\varphi_{a},\varphi_{a}\backslash\bar{v})\times
(X^{+},S))\\
&  =H_{\ast}(\varphi_{a},\varphi_{a}\backslash\bar{v})\otimes H_{\ast}%
(X^{+},S)\\
&  =H_{\ast-\mu}(\varphi_{a},\varphi_{a}\backslash\bar{v})=C_{\ast-\mu
}(\varphi,\bar{v})\text{,}%
\end{align*}
where we have used the fact that $H_{q}(X^{+},S)=\delta_{q,\mu}\mathcal{G}$,
since $\dim X^{+}=\mu$.
\end{proof}

In critical point theory, it will be very convenient if the gradient
of the functional under consideration is a compact perturbation of
the identity operator. Hence, if the original functional $f$ has
this property, we hope that the reduced functional $\varphi$ also
has such property. This is true if $\nabla f:X\rightarrow X$ maps
bounded sets to bounded sets.

\begin{proposition}
[{\cite[Corollary 2.2]{MR2488059}}]\label{p2}Let $X$ be a separable
Hilbert space and $f\in C^{1}( X,\mathsf{R}) $. Assume $( A_{+}) $
or $( A_{-}) $ holds. If $\nabla f:X\rightarrow X$ is bounded and
there is a compact operator $K:X\rightarrow X$ such that $\nabla
f=\vec{1}_{X}-K$, then there is a compact operator
$Q:X^{-}\rightarrow X^{-}$ such that $\nabla\varphi=\vec{1}_{(
X^{-}) }-Q$.
\end{proposition}

Finally, for the convenience of our later application, we recall a
homological version of the famous three critical points theorem.

\begin{proposition}
[{\cite[Theorem 2.1]{MR1828101}}]\label{ls}Let $X$ be a Banach space
and $f\in C^{1}(X,\mathsf{R})$ satisfy the Palais-Smale $(PS)$
condition. Assume that $f$ is bounded from below. If
$C_{\ell}(f,\vec{0})\neq0$ for some $\ell\neq0$, then $f$ has at
least three critical points.
\end{proposition}

We note that according to \cite[Page 33]{MR1196690}, $C_{\ell}(
f,\vec{0}) \neq0$ for some $\ell\neq0$ implies that $\vec{0}$ is not
a local minimizer of $f$, as required in the original statement of
\cite[Theorem 2.1]{MR1828101}.

\section{Multiple solutions of elliptic systems}

In this section, as application of our abstract results on critical
groups, we
consider elliptic gradient systems of the form
\begin{equation}
\left\{
\begin{array}
[c]{ll}%
-\Delta u=F_{u}( x,u,v)\text{,}\qquad & \text{in }\Omega\text{,}\\
-\Delta v=F_{v}( x,u,v)\text{,} & \text{in }\Omega\text{,}\\
\phantom{-\Delta }u=v=0\text{,} & \text{on }\partial\Omega\text{,}%
\end{array}
\right.  \label{se1}%
\end{equation}
where $\Omega\subset\mathsf{R}^{N}$ is a bounded smooth domain,
$F\in C^{1}(
\Omega\times\mathsf{R}^{2},\mathsf{R}) $ satisfies the linear growth condition%
\begin{equation}
\left\vert \nabla F( x,\vec{z}) \right\vert \leq\Lambda\left\vert
\vec {z}\right\vert \text{,\qquad}( x,\vec{z})
\in\Omega\times\mathsf{R}^{2}
\label{se2}%
\end{equation}
for some constant $\Lambda>0$. Here, to simplify the notations we
denote $\vec{z}=( u,v) $. The gradient is taken with respect to
$\vec{z}$. Without lost of generality we may assume $F( x,\vec{0})
=0$. We also assume $\nabla F( x,\vec{0}) =\vec{0}$, so that
$\vec{z}=\vec{0}$ is a trivial solution of \eqref{se1}. Therefore,
we will focus on nontrivial solutions.

To apply variational methods, let $X$ be the Hilbert space $H_{0}^{1}%
(\Omega)\times H_{0}^{1}(\Omega)$ endowed with the inner product%
\[
\left\langle \vec{z},\vec{w}\right\rangle =\int_{\Omega}\nabla\vec{z}%
\cdot\nabla\vec{w}\,\mathrm{d}x
\]
and corresponding norm $\Vert\cdot\Vert$, here $\vec{z}=(u,v)$,
$\nabla\vec {z}=(\nabla u,\nabla v)$, the dot `$\cdot$' represents
the standard inner product in $\mathsf{R}^{2}$. Under the growth
condition \eqref{se2}, the
functional $\Phi:X\rightarrow\mathsf{R}$,%
\begin{equation}
\Phi(\vec{z})=\frac{1}{2}\int_{\Omega}\left\vert
\nabla\vec{z}\right\vert^2
\mathrm{d}x-\int_{\Omega}F(x,\vec{z})\mathrm{d}x \label{sfun}%
\end{equation}
is well defined and of class $C^{1}$. The critical points of $\Phi$
are solutions of the system \eqref{se1}.

Before state our assumptions on $F$ and our main results, let us
denote by $\mathcal{M}_{2}(\Omega)$ the set of those positive definite symmetric matrix
functions $A:\bar{\Omega
}\rightarrow M_{2\times2}(\mathsf{R})$ whose entries are continuous real functions on $\bar{\Omega}$.

For given $A\in\mathcal{M}_{2}(\Omega)$, there is an associated
weighted
eigenvalue problem%
\begin{equation}
\left\{
\begin{array}
[c]{ll}%
-\Delta\vec{z}=\lambda A(x)\vec{z} & \text{in }\Omega\text{,}\\
\vec{z}=\vec{0} & \text{on }\partial\Omega\text{.}%
\end{array}
\right.  \label{se3}%
\end{equation}
Since $A(x)$ is positive definite, using the spectral theory
of compact self-adjoint operator, it is well known that there is a complete list of distinct eigenvalues%
\[
0<\lambda_{1}(A)<\lambda_{2}(A)<\cdots
\]
such that $\lambda_{n}(A)\rightarrow+\infty$ as
$n\rightarrow\infty$.
Our multiplicity results depend on the interaction between the
nonlinearity $F$ and the eigenvalues of \eqref{se3}.

We also need the following concepts introduced by da Silva
\cite[Definition 1.4]{MR2651745}.

\begin{definition}
\label{sd1}Let $A,B\in\mathcal{M}_{2}( \Omega) $.

\begin{description}
\item[\rm(i)] We define $A\leq B$ if $A(x)\vec{z}\cdot\vec{z}\leq B(x)\vec{z}\cdot
\vec{z}$ for all $( x,\vec{z}) \in\Omega\times\mathsf{R}^{2}$.

\item[\rm(ii)] We define $A\preceq B$ if $A\leq B$ and $B-A$ is positive definite on
$\tilde{\Omega}\subset\Omega$ with $\left\vert
\tilde{\Omega}\right\vert >0$, where $\left\vert \cdot\right\vert $
stands for the Lebesgue measure.
\end{description}
\end{definition}

To state our assumptions on the nonlinearity $F(x,\vec{z})$, we
assume that there exists $A_{0}\in\mathcal{M}_{2}(\Omega)$ with
$\lambda_{m}(A_{0})=1$ for
some $m\in\mathsf{N}$, such that%
\begin{equation}
G(x,\vec{z}):= F(x,\vec{z})-\frac{1}{2}A_{0}(x)\vec{z}\cdot\vec{z}%
=o(\left\vert \vec{z}\right\vert ^{2})\text{,\qquad as }\left\vert
\vec
{z}\right\vert \rightarrow0 \label{se4}%
\end{equation}
uniformly for $x\in\Omega$. We then assume the following conditions
on $F$.

\begin{description}
\item[$( F_{0}^{\pm}) $] There exists some $\delta>0$ such that $\pm G( x,\vec{z})
>0$ for $0<\left\vert \vec{z}\right\vert \leq\delta$.

\item[$(F_{\infty}^{\pm})$] There exists \ $A_{\infty}\in\mathcal{M}%
_{2}(\Omega)$ with $\lambda_{k}(A_{\infty})=1$ for some
$k\in\mathsf{N}$, such
that%
\[
\lim_{\left\vert \vec{z}\right\vert \rightarrow\infty}\left(
F(x,\vec
{z})-\frac{1}{2}A_{\infty}(x)\vec{z}\cdot\vec{z}\right)  =\pm\infty\text{.}%
\]

\end{description}

\begin{remark}
\begin{description}
\item[\rm(i)] If $F\in C^{2}(\Omega\times\mathsf{R}^{2},\mathsf{R})$ verifies
$F(x,\vec{0})=0$, $\nabla F(x,\vec{0})=\vec{0}$, then by the Taylor
formular we see that \eqref{se4} holds with $A_{0}(x)$ being the
Hessian of $F(x,\cdot)$ at $\vec{z}=\vec{0}$.

\item[\rm(ii)] If $1\in( \lambda_{m}(A_{0}),\lambda_{m+1}(A_{0})) $ for some
$m\in\mathsf{N}$, namely the problem \eqref{se1} is non resonant at
the origin, then the condition $(F_{0}^{\pm})$ can be removed. The
same remark applies to $(F_{\infty}^{\pm})$ if \eqref{asl} holds.
\end{description}
\end{remark}

For the sake of simplicity, we denote $\lambda_{n}(A_{0})$ by $\lambda_{n}%
^{0}$, and $\lambda_{n}(A_{\infty})$ by $\lambda_{n}^{\infty}$. Set%
\[
d_{n}^{0}=\sum_{i=1}^{n}\dim\ker(-\Delta-\lambda_{i}^{0}A_{0})\text{,\qquad}%
d_{n}^{\infty}=\sum_{i=1}^{n}\dim\ker(-\Delta-\lambda_{i}^{\infty}A_{\infty
})\text{.}%
\]
Our main results are the following theorems

\begin{theorem}
\label{st1}Suppose that $F\in
C^{1}(\Omega\times\mathsf{R}^2,\mathsf{R})$ satisfies \eqref{se2} and
$(F_{\infty}^{+})$. Suppose moreover that there exists
$\beta\in\mathcal{M}_{2}(\Omega)$,
$\beta\preceq\lambda_{k+1}^{\infty
}A_{\infty}$ such that%
\begin{equation}
(\nabla F(x,\vec{z}_{1})-\nabla F(x,\vec{z}_{2}))\cdot(\vec{z}_{1}-\vec{z}%
_{2})\leq\beta(x)(\vec{z}_{1}-\vec{z}_{2})\cdot(\vec{z}_{1}-\vec{z}%
_{2})\text{,} \label{se5}%
\end{equation}
then the system \eqref{se1} has at least two nontrivial solutions in
each of the following cases:

\begin{description}
\item[\rm(i)] $( F_{0}^{+}) $ holds with $d_{m}^{0}\neq d_{k}^{\infty}$.

\item[\rm(ii)] $( F_{0}^{-}) $ holds with $d_{m-1}^{0}\neq d_{k}^{\infty}$.
\end{description}
\end{theorem}

\begin{theorem}
\label{st2}Suppose that $F\in
C^{1}(\Omega\times\mathsf{R}^2,\mathsf{R})$ satisfies \eqref{se2} and
$(F_{\infty}^{-})$. Suppose moreover that there exists
$\beta\in\mathcal{M}_{2}(\Omega)$,
$\beta\succeq\lambda_{k-1}^{\infty
}A_{\infty}$ such that%
\begin{equation}
(\nabla F(x,\vec{z}_{1})-\nabla F(x,\vec{z}_{2}))\cdot(\vec{z}_{1}-\vec{z}%
_{2})\geq\beta(x)(\vec{z}_{1}-\vec{z}_{2})\cdot(\vec{z}_{1}-\vec{z}%
_{2})\text{,} \label{se6}%
\end{equation}
then the system \eqref{se1} has at least two nontrivial solutions in
each of the following cases:

\begin{description}
\item[\rm(i)] $( F_{0}^{+}) $ holds with $d_{m}^{0}\neq d_{k-1}^{\infty}$.

\item[\rm(ii)] $( F_{0}^{-}) $ holds with $d_{m-1}^{0}\neq d_{k-1}^{\infty}$.
\end{description}
\end{theorem}

Obviously, if in addition to \eqref{se4} we have
\begin{equation}
\left\vert \nabla F(x,\vec{z})-A_{\infty}(x)\vec{z}\right\vert
=o(\left\vert \vec{z}\right\vert )\text{,\qquad as }\left\vert
\vec{z}\right\vert \rightarrow\infty\text{,}\label{asl}
\end{equation}
then \eqref{se2} holds. In this case we say that the problem
\eqref{se1} is asymptotically linear at infinity. Since the pioneer
work of Amann and Zehnder \cite{MR600524}, asymptotically linear
problems for a single equation have captured great interest. We
referr to \cite{MR1730881,MR1662078} and references therein for some
interesting results.

The asymptotically linear elliptic systems have captured some
attentions in recent years. In \cite{MR1813819,MR1697052}, the
authors considered the case that $A_{0}=A_{\infty}$ are constant
matrices. In \cite{MR2227915}, Furtado and Paiva studied the case
that $A_{0}$ and $A_{\infty}$ are variable matrices. Under some
conditions that ensure the Euler-Lagrange functional $\Phi$
satisfying the Ceremi type compactness condition, they obtained a
nontrivial solution.

In \cite{MR2651745}, also for the case that $A_{0}$ and $A_{\infty}$
are variable matrices, Silva obtained two nontrivial solutions for
the problem by applying Morse index type argument to $\Phi$. Thus it
is essential to require $F\in C^{2}$ so that $\Phi$ is also of class
$C^{2}$. Similar to \cite{MR1662078}, the reduction conditions
\eqref{se5} and \eqref{se6} are used to control the Morse index.
Finally, he also required some conditions to guarantee Ceremi type
compactness for $\Phi$.

Recently, Furtado and Paiva obtained a multiplicity result
\cite[Theorem
1.1]{MR2685145} under \eqref{se5} and $(F_\infty^+)$. But they only considered the case that $1\in( \lambda_{m}%
(A_{0}),\lambda_{m+1}(A_{0})) $ for some $m\in\mathsf{N}$, and they
also required $F\in C^{2}$. Therefore, our Theorem \ref{st1} is an
improvement.

Under our assumptions the functional $\Phi$ may not satisfy the $(
PS) $ condition. To overcome this difficulty, as in
\cite{MR2488059,MR1781225}, we will perform the saddle point
reduction and turn to consider the reduced functional $\varphi$. It
turns out that the case of Theorem \ref{st2} is more difficult,
because the reduced functional $\varphi$ is defined on an infinite
dimentional subspace.

Although our proof of Theorem \ref{st2} is based on some idea from
\cite{MR2488059}, there are some real differences. The most
significant one is that if $A_{0}$ and $A_{\infty}$ are different
matrix functions, the local linking approach used in
\cite{MR2488059} does not work any more, see Remark \ref{rek2} for
details. Hence we must use critical groups and our abstract result (Theorem \ref{t1}) is
crucial.

\section{Proofs of Theorems \ref{st1} and \ref{st2}}

Let $\Phi$ be the functional introduced in \eqref{sfun}. To prove
our theorems it suffices to show that $\Phi$ has two nonzero
critical points.

Recall that for our Hilbert space $X$ and $p\in\left[
2,2^{\ast}\right]  $,
by Sobolev inequality there is a constant $S_{p}$ such that%
\begin{equation}
\left\vert \vec{z}\right\vert _{p}:=\left(  \int_{\Omega}\left\vert
\vec {z}\right\vert ^{p}\mathrm{d}x\right)  ^{1/p}\leq
S_{p}\left\Vert \vec{z}\right\Vert
\text{.}\label{sSob}%
\end{equation}
That is, the embedding $X\hookrightarrow L^{p}(\Omega)\times
L^{p}(\Omega)$ is continuous. Moreover, using the Rellich-Kondrachov
theorem we see that the embedding is also compact if
$p\in\lbrack2,2^{\ast})$.

\begin{lemma}
\label{sl2}

\begin{description}
\item[\rm(i)] If $( F_{0}^{+}) $ holds, then $C_{d_{m}^{0}}( \Phi,\vec{0}) \neq0$.

\item[\rm(ii)] If $( F_{0}^{-}) $ holds, then $C_{d_{m-1}^{0}}( \Phi,\vec{0}) \neq0$.
\end{description}
\end{lemma}

\begin{proof}
We only prove Case (ii). Set $V_{0}=\ker( -\Delta-\lambda_{m}^{0}A_{0}) $,%
\[
V_{-}=\bigoplus_{i=1}^{m-1}\ker( -\Delta-\lambda_{i}^{0}A_{0})
\text{,\qquad }V_{+}=\overline{\bigoplus_{i=m+1}^{\infty}\ker(
-\Delta-\lambda_{i}^{0}A_{0})
}\text{.}%
\]
Then $\dim V_{-}=d_{m-1}^{0}$. We will show that $\Phi$ has a local
linking with respect to the decomposition $X=V_{-}\oplus(
V_{0}\oplus V_{+})
$.\ Namely, there exists $\rho>0$ such that%
\begin{equation}
\left\{
\begin{array}
[c]{l}%
\Phi( \vec{z}) \leq0\text{\qquad for }\vec{z}\in V_{-}\text{,
}\left\Vert
\vec{z}\right\Vert \leq\rho\text{,}\\
\Phi( \vec{z}) >0\text{\qquad for }\vec{z}\in V_{0}\oplus
V_{+}\text{,
}0<\left\Vert \vec{z}\right\Vert \leq\rho\text{.}%
\end{array}
\right.  \label{slo}%
\end{equation}
The desired result will then follow from \cite[Theorem
2.1]{MR1110119}. To prove \eqref{slo}, we argue as in \cite[Page
24]{MR1312028}.

Since $\lambda_{m}^{0}=1$ is an isolated eigenvalue, it is well know
that
there exists positive number $\kappa>0$ such that%
\[
\pm\frac{1}{2}\int_{\Omega}\left(  \left\vert
\nabla\vec{z}\right\vert ^{2}-A_{0}(x)\vec{z}\cdot\vec{z}\right)
\mathrm{d}x\geq\kappa\left\Vert \vec
{z}\right\Vert ^{2}\text{,\qquad}z\in V_{\pm}\text{.}%
\]
Using \eqref{se2}, \eqref{se4}, there exists $C>0$ such that%
\begin{equation}
\left\vert G(x,\vec{z})\right\vert
\leq\frac{\kappa}{8S_{2}^{2}}\left\vert
\vec{z}\right\vert ^{2}+C_{1}\left\vert \vec{z}\right\vert ^{2^{\ast}%
}\text{,\qquad}(x,\vec{z})\in\Omega\times\mathsf{R}^{2}\text{.}\label{sleq}%
\end{equation}
For $\vec{z}\in V_{-}$, using \eqref{sleq} and \eqref{sSob}, we obtain%
\begin{align}
\Phi(\vec{z}) &  =\frac{1}{2}\int_{\Omega}\left(  \left\vert
\nabla\vec {z}\right\vert ^{2}-A_{0}(x)\vec{z}\cdot\vec{z}\right)
\mathrm{d}x-\int_{\Omega
}G(x,\vec{z})\mathrm{d}x\nonumber\\
&  \leq-\kappa\left\Vert \vec{z}\right\Vert ^{2}+\frac{\kappa}{8S_{2}^{2}%
}\left\vert \vec{z}\right\vert _{2}^{2}+C_{1}\left\vert
\vec{z}\right\vert
_{2^{\ast}}^{2^{\ast}}\nonumber\\
&  \leq-\frac{\kappa}{2}\left\Vert \vec{z}\right\Vert
^{2}+C_{3}\left\Vert \vec{z}\right\Vert
^{2^{\ast}}\leq0\text{,}\label{sv1}%
\end{align}
provided $\Vert
\vec{z}\Vert\le\rho_1=\left(2^{-1}C_3^{-1}\kappa\right)^{1/(2^*-2)}$.

On the other hand, since $\dim V_{0}<\infty$, there exists $C_{2}>0$ such that%
\[
\left\vert \vec{v}\right\vert _{\infty}\leq C_{2}\left\Vert
\vec{v}\right\Vert
\text{,\qquad for }\vec{v}\in V_{0}\text{.}%
\]
For $\vec{z}\in V_{0}\oplus V_{+}$ with $\left\Vert
\vec{z}\right\Vert \leq2^{-1}C_{2}^{-1}\delta$, we may write
$\vec{z}=\vec{v}+\vec{w}$, where
$\vec{v}\in V_{0}$, $\vec{w}\in V_{+}$. Set%
\[
\Omega_{1}=\left\{  x\in\Omega\left\vert \,\left\vert
\vec{w}(x)\right\vert \leq\frac{\delta}{2}\right.  \right\}
\text{,\qquad}\Omega_{2}=\Omega
\backslash\Omega_{1}\text{.}%
\]
For $x\in\Omega_{2}$, we have%
\begin{align*}
\left\vert \vec{z}(x)\right\vert  &  \leq\left\vert
\vec{v}(x)\right\vert +\left\vert \vec{w}(x)\right\vert
\leq\left\vert \vec{v}\right\vert _{\infty
}+\left\vert \vec{w}(x)\right\vert \\
&  \leq C_{2}\left\Vert \vec{v}\right\Vert +\left\vert
\vec{w}(x)\right\vert
\\
&  \leq C_{2}\left\Vert \vec{z}\right\Vert +\left\vert
\vec{w}(x)\right\vert \leq\frac{\delta}{2}+\left\vert
\vec{w}(x)\right\vert \leq2\left\vert \vec
{w}(x)\right\vert \text{.}%
\end{align*}
By \eqref{sleq}, we see that for $x\in\Omega_{2}$,%
\[
G(x,\vec{z})\leq\frac{\kappa}{8S_{2}^{2}}\left\vert
\vec{z}\right\vert
^{2}+C_{1}\left\vert \vec{z}\right\vert ^{2^{\ast}}\leq\frac{\kappa}%
{2S_{2}^{2}}\left\vert \vec{w}\right\vert
^{2}+C_{1}^{\prime}\left\vert
\vec{w}\right\vert ^{2^{\ast}}\text{.}%
\]
This is also true for $x\in\Omega_{1}$, because in this case%
\[
\left\vert \vec{z}(x)\right\vert \leq\left\vert
\vec{v}(x)\right\vert +\left\vert \vec{w}(x)\right\vert
\leq\left\vert \vec{v}\right\vert _{\infty
}+\frac{\delta}{2}\leq C_{2}\left\Vert \vec{v}\right\Vert +\frac{\delta}%
{2}\leq C_{2}\left\Vert \vec{z}\right\Vert +\frac{\delta}{2}\leq\delta\text{,}%
\]
hence $G(x,\vec{z})\leq0$ by our assumption $(F_{0}^{-})$. Therefore
using \eqref{sSob} we deduce
\begin{align}
\Phi(\vec{z}) &  =\frac{1}{2}\int_{\Omega}\left(  \left\vert
\nabla\vec {w}\right\vert ^{2}-A_{0}(x)\vec{w}\cdot\vec{w}\right)
\mathrm{d}x-\int_{\Omega
}G(x,\vec{z})\mathrm{d}x\nonumber\\
&  \geq\kappa\left\Vert \vec{w}\right\Vert ^{2}-\frac{\kappa}{2S_{2}^{2}%
}\left\vert \vec{w}\right\vert _{2}^{2}-C_{1}^{\prime}\left\vert
\vec
{w}\right\vert ^{2^{\ast}}\nonumber\\
&  \geq\frac{\kappa}{2}\left\Vert \vec{w}\right\Vert
^{2}-C_{4}\left\Vert \vec{w}\right\Vert ^{2^{\ast}}\text{,\qquad}\vec{z}=\vec{v}%
+\vec{w}\in V_{0}\oplus V_{+}\text{.}\label{sv2}%
\end{align}

Now, let $\vec{z}\in V_{0}\oplus V_{+}$ be such that
\[
0<\Vert
\vec{z}\Vert\le\rho_2=\min\left\{\frac{\delta}{2C_2},\left(\frac{\kappa}{2C_4}\right)^{1/(2^*-2)}\right\}\text{.}
\]
If $\vec{w}\ne0$, since $\Vert \vec{w}\Vert\le\Vert \vec{z}\Vert$,
by \eqref{sv2} we may deduce $\Phi(\vec{z})>0$. If $\vec{w}=0$, then
$\vec{z}\in V_0$ and $\vert\vec{z}\vert_\infty\le
C_2\Vert\vec{z}\Vert\le\delta$, using $(F_0^-)$ again, we also have
\[
\Phi(\vec{z})=-\int_\Omega
G(x,\vec{z})\mathrm{d}x=-\int_{\vert\vec{z}\vert\le\delta}
G(x,\vec{z})\mathrm{d}x>0\text{.}
\]
Combining the above argument, we see that \eqref{slo} is true with
$\rho=\min\{\rho_1,\rho_2\}$.
\end{proof}

\begin{remark}
If we replace $(  F_{0}^{-})  $ by the weaker condition: $G(
x,\vec{z})  \leq0$ for $\left\vert \vec{z}\right\vert \leq\delta$,
then we can only obtain $\Phi(  \vec{z}) \geq0$ in the second line
of \eqref{slo}. Namely, we only have a weak local linking in the
sense of Brezis and Nirenberg \cite{MR1127041}, which is not
sufficient to obtain Lemma \ref{sl2} via the result of J.Q. Liu
\cite{MR1110119}. However, if \eqref{se2} is  replaced by the
stronger condition%
\[
\left\vert \nabla F(  x,\vec{z}_{1})  -\nabla F( x,\vec
{z}_{2})  \right\vert \leq\Lambda\left\vert \vec{z}_{1}-\vec{z}%
_{2}\right\vert \text{,\qquad}x\in\Omega\text{,\quad}\vec{z}_{1},\vec{z}_{2}%
\in\mathsf{R}^{2}\text{,}%
\]
then the functional $\Phi$ is of class $C^{2-0}$. According to
Perera \cite[Theorem 2.6]{MR1749421}, we can still obtain the
conclusion of Lemma \ref{sl2}.
\end{remark}

As mentioned before, the proof of Theorem \ref{st2} is more
difficult. Therefore, in what follows we will only prove Theorem
\ref{st2}. Let%
\begin{equation}
X^{-}=\overline{\bigoplus_{i=k}^{\infty}\ker(
-\Delta-\lambda_{i}^{\infty }A_{\infty})
}\text{,\qquad}X^{+}=\bigoplus_{i=1}^{k-1}\ker( -\Delta
-\lambda_{i}^{\infty}A_{\infty}) \text{.}\label{red}
\end{equation}
To verify the condition $( A_{-}) $ and perform saddle point
reduction, we need the following result.

\begin{proposition}
[{\cite[Proposition 3.9 (b)]{MR2651745}}]\label{sp3}Let
$\beta\in\mathcal{M}_{2}( \Omega) $,
$\beta\succeq\lambda_{k-1}^{\infty}A_{\infty}$. Then there exists
$\delta>0$ such that%
\[
-\left\Vert \vec{z}\right\Vert
^{2}+\int_{\Omega}\beta(x)\vec{z}\cdot\vec
{z}\,\mathrm{d}x\geq\delta\left\Vert \vec{z}\right\Vert
^{2}\text{,\qquad for all }\vec
{z}\in X^{+}\text{.}%
\]

\end{proposition}

\begin{remark}
There is also a similar result for the dual case
$\beta\preceq\lambda _{k+1}^{\infty}A_{\infty}$, see
\cite[Proposition 3.9 (a)]{MR2651745}. This will be needed in the
proof of Theorem \ref{st1}.
\end{remark}

Now, for $\vec{v}\in X^{-}$ and $\vec{w}_{1},\vec{w}_{2}\in X^{+}$,
using
Proposition \ref{sp3} and our assumption \eqref{se6} we obtain%
\begin{align*}
  -&\left\langle \nabla\Phi( \vec{v}+\vec{w}_{1}) -\nabla\Phi(
\vec{v}+\vec
{w}_{2}) ,\vec{w}_{1}-\vec{w}_{2}\right\rangle \\
&  =-\int_{\Omega}\left\vert \nabla( \vec{w}_{1}-\vec{w}_{2})
\right\vert ^{2}\mathrm{d}x+\int_{\Omega}( \nabla F(
x,\vec{v}+\vec{w}_{1}) -\nabla F( x,\vec
{v}+\vec{w}_{2}) ) \cdot( \vec{w}_{1}-\vec{w}_{2}) \mathrm{d}x\\
&  \geq-\int_{\Omega}\left\vert \nabla( \vec{w}_{1}-\vec{w}_{2})
\right\vert
^{2}\mathrm{d}x+\int_{\Omega}\beta(x)( \vec{w}_{1}-\vec{w}_{2}) \cdot( \vec{w}_{1}%
-\vec{w}_{2}) \mathrm{d}x\\
&  \geq\delta\left\Vert \vec{w}_{1}-\vec{w}_{2}\right\Vert ^{2}\text{.}%
\end{align*}
Therefore, $\Phi$ satisfies the condition $(A_{-})$ and we obtain a
reduced functional $\varphi:X^{-}\rightarrow\mathsf{R}$, which is of
class $C^{1}$. It suffices to find two non-zero critical points of
$\varphi$.

We want to show that $\varphi$ is coercive. For this, it is quite
natural to pick a sequence $\left\{  \vec{v}_{n}\right\}  $ in
$X^{-}$ such that $\left\Vert \vec{v}_{n}\right\Vert
\rightarrow\infty$. Then consider the
normalization sequence $\big\{  \Vert \vec{v}_{n}\Vert ^{-1}%
\vec{v}_{n}\big\}  $. However, since $\dim X^{-}=\infty$, the weak
limit of the normalization sequence may be the zero element in
$X^{-}$. This makes it difficult to prove that $\varphi(
\vec{v}_{n}) \rightarrow+\infty$.

To get around this difficulty, as in \cite{MR2488059} we consider
$\Phi_{1}$, the restriction of $\Phi$ on $X^{-}$. Then $\Phi_{1}\in
C^{1}(X^{-},\mathsf{R})$. The following `non vanishing lemma' is the
key ingredient of our approach.

\begin{lemma}
\label{sl4}Let $\left\{  \vec{v}_{n}\right\}  $ be a sequence in
$X^{-}$ such that $\Phi_{1}( \vec{v}_{n}) \leq c$ and $\left\Vert
\vec{v}_{n}\right\Vert \rightarrow\infty$. Denote
$\vec{v}_{n}^{0}=\left\Vert \vec{v}_{n}\right\Vert
^{-1}\vec{v}_{n}$. Then there is a subsequence of $\left\{  \vec{v}_{n}%
^{0}\right\}  $ which converges weakly to some point
$\vec{v}^{0}\neq\vec{0}$.
\end{lemma}

\begin{proof}
The proof is quite similar to that of \cite[Lemma 3.2]{MR2488059},
where instead of $\Phi_{1}( \vec{v}_{n}) \leq c$, it is assumed that
$\nabla\Phi_{1}( \vec {v}_{n}) \rightarrow0$. Since $\left\{
\vec{v}_{n}^{0}\right\}  $ is bounded, up to a subsequence, we may
assume that $\vec{v}_{n}^{0}\rightharpoonup\vec
{v}^{0}$ in $X^{-}$. The compactness of the embedding%
\[
X^{-}\hookrightarrow X\hookrightarrow L^{2}( \Omega) \times L^{2}(
\Omega)
\]
implies that $\vec{v}_{n}^{0}\rightarrow\vec{v}^{0}$ in $L^{2}(
\Omega) \times
L^{2}( \Omega) $. By \eqref{se2} we have%
\[
\left\vert F( x,\vec{z}) \right\vert
\leq\frac{1}{2}\Lambda\left\vert \vec
{z}\right\vert ^{2}\text{,\qquad}( x,\vec{z}) \in\Omega\times\mathsf{R}%
^{2}\text{.}%
\]
Therefore
\begin{align*}
2c \geq2\Phi_{1}( \vec{v}_{n})  &  =\int_{\Omega}\left\vert \nabla\vec{v}%
_{n}\right\vert ^{2}\mathrm{d}x-\int_{\Omega}2F( x,\vec{v}_{n}) \mathrm{d}x\\
&  \geq\int_{\Omega}\left\vert \nabla\vec{v}_{n}\right\vert
^{2}\mathrm{d}x-\Lambda
\int_{\Omega}\left\vert \vec{v}_{n}\right\vert ^{2}\mathrm{d}x\\
&  =\left\Vert \vec{v}_{n}\right\Vert ^{2}-\Lambda\left\vert \vec{v}%
_{n}\right\vert _{2}^{2}\text{.}%
\end{align*}
Multiplying by $\left\Vert \vec{v}_{n}\right\Vert ^{-2}$ on both
sides, we
deduce%
\[
2c\left\Vert \vec{v}_{n}\right\Vert ^{-2}\geq1-\Lambda\vert\vec{v}_{n}^{0}%
\vert_{2}^{2}\text{.}%
\]
Since
$\vert\vec{v}_{n}^{0}\vert_{2}\rightarrow\vert\vec{v}^{0}\vert_{2}$
and $\left\Vert \vec{v}_{n}\right\Vert ^{-2}\rightarrow0$, the above
inequality implies that
$\vert\vec{v}^{0}\vert_{2}^{2}\geq\Lambda^{-1}$ and hence
$\vec{v}^{0}\neq\vec{0}$.
\end{proof}

\begin{remark}
This is the only place where we need \eqref{se2}. In the case of
Theorem \ref{st1}, the reduced functional $\varphi$ is defined on a
finite dimensional subspace. Hence it is easy to obtain the
coerciveness of $\varphi$ using the assumption $(F_\infty^+)$.
Therefore, in Theorem \ref{st1}, we may replace \eqref{se2} with a
subcritical growth condition.
\end{remark}

\begin{lemma}
\label{sll3}The functional $\Phi_{1}:X^{-}\rightarrow\mathsf{R}$ is
coercive, and bounded from below.
\end{lemma}

\begin{proof}
Assume for a contradiction that for some $\left\{
\vec{v}_{n}\right\} \subset X^{-}$ and $c>0$ we have
\begin{equation}
\Phi_{1}(\vec{v}_{n})\leq c\text{,\qquad}\left\Vert
\vec{v}_{n}\right\Vert
\rightarrow\infty\text{.}\label{sx}%
\end{equation}
Let $\vec{v}_{n}^{0}=\left\Vert \vec{v}_{n}\right\Vert ^{-1}v_{n}$,
by Lemma \ref{sl4}, up to a subsequence we have
$\vec{v}_{n}^{0}\rightharpoonup\vec
{v}^{0}$ for some $\vec{v}^{0}\neq\vec{0}$. Let%
\[
\Theta=\left\{  \left.  x\in\Omega\right\vert
\,\vec{v}^{0}(x)\neq0\right\}
\text{,}%
\]
then $\left\vert \Theta\right\vert >0$. For $x\in\Theta$ we have%
\[
\left\vert \vec{v}_{n}(x)\right\vert =\left\Vert
\vec{v}_{n}\right\Vert
\left\vert \vec{v}_{n}^{0}(x)\right\vert \rightarrow\infty\text{.}%
\]
By $(F_{\infty}^{-})$ and the Fatou Lemma,%
\[
\int_{\Theta}\left(  \frac{1}{2}A_{\infty}(x)\vec{v}_{n}\cdot\vec{v}%
_{n}-F(x,\vec{v}_{n})\right)
\mathrm{d}x\rightarrow+\infty\text{,\qquad as
}n\rightarrow\infty\text{.}%
\]

On the other hand, $(F_{\infty}^{-})$ also implies the existence of
$M>0$ such
that%
\begin{equation}
\frac{1}{2}A_{\infty}(x)\vec{z}\cdot\vec{z}-F(x,\vec{z})\geq-M\text{,\qquad
}(x,\vec{z})\in\Omega\times\mathsf{R}^{2}\text{.}\label{see}%
\end{equation}
Therefore,%
\begin{align*}
\Phi_{1}(\vec{v}_{n}) &  =\frac{1}{2}\int_{\Omega}\left\vert \nabla\vec{v}%
_{n}\right\vert ^{2}\mathrm{d}x-\int_{\Omega}F(x,\vec{v}_{n})\mathrm{d}x\\
&  \geq\int_{\Omega}\left(
\frac{1}{2}A_{\infty}(x)\vec{v}_{n}\cdot\vec
{v}_{n}-F(x,\vec{v}_{n})\right)  \mathrm{d}x\\
&  =\left(  \int_{\Theta}+\int_{\Omega\backslash\Theta}\right)
\left(
\frac{1}{2}A_{\infty}(x)\vec{v}_{n}\cdot\vec{v}_{n}-F(x,\vec{v}_{n})\right)
\mathrm{d}x\\
&  \geq\int_{\Theta}\left(
\frac{1}{2}A_{\infty}(x)\vec{v}_{n}\cdot\vec
{v}_{n}-F(x,\vec{v}_{n})\right)  \mathrm{d}x-M\left\vert
\Omega\backslash
\Theta\right\vert \rightarrow+\infty\text{.}%
\end{align*}
This contradicts with \eqref{sx}. Thus $\Phi_{1}$ is coercive. It
follows that $\Phi_{1}$ is bounded from below.
\end{proof}

\begin{remark}\label{rek1}
In \cite[Lemma 3.2]{MR2488059}, another version of `non vanishing
lemma' (as mentioned in the proof of Lemma \ref{sl4}) is proved and
used in \cite[Lemma 3.4]{MR2488059} to show that $\Phi_{1}$
satisfies the $( PS) $ condition. Then the coerciveness of
$\Phi_{1}$ is obtained via a well-known result of Li \cite{Li86},
see also \cite{MR1044221}. Our argument in Lemmas \ref{sl4} and
\ref{sll3} does not involve the derivative information of
$\Phi_{1}$, hence is considerably simpler.
\end{remark}

\begin{lemma}
\label{sl3}Under the assumption of Theorem \ref{st2}, the functional
$\varphi$ is bounded from below. Moreover, $\varphi$ satisfies the
$( PS) $ condition.
\end{lemma}

\begin{proof}
Let $K:X\rightarrow X$ be defined as%
\[
\left\langle K\vec{z},\vec{w}\right\rangle =\int_{\Omega}\nabla F(
x,\vec{z})
\cdot\vec{w}\,\mathrm{d}x\text{.}%
\]
Then $K$ is compact and $\nabla\Phi=\vec{1}_{X}-K$. Obviously
$\nabla\Phi$ maps bounded sets to bounded sets. By Proposition
\ref{p2}, $\nabla\varphi$ is also a compact perturbation of
$\vec{1}_{(X^{-})}$.

By the definition of the reduced functional $\varphi$, we have%
\[
\varphi( \vec{v}) =\max_{\vec{w}\in X^{+}}\Phi( \vec{v}+\vec{w})
\geq\Phi(
\vec{v}) =\Phi_{1}( \vec{v}) \text{.}%
\]
Using Lemma \ref{sll3} we see that $\varphi$ is also coercive and
bounded from below. In particular, any $( PS) $ sequence of
$\varphi$ is bounded. Applying \cite[Proposition 2.2]{MR1411681}, we
deduce that $\varphi$ satisfies $( PS) $.
\end{proof}

\begin{proof}
[Proof of Theorem \ref{st2}]We prove the case (i). By Lemma
\ref{sl3},
$\varphi$ satisfies the $( PS) $ condition, and bounded from below. Note that%
\[
\mu=\dim X^{+}=d_{k-1}^{\infty}\text{,}%
\]
by Theorem \ref{t1} and Lemma \ref{sl2} we obtain%
\[
C_{d_{m}^{0}-d_{k-1}^{\infty}}( \varphi,\vec{0}) \cong
C_{d_{m}^{0}}(
\Phi,\vec{0}) \neq0\text{.}%
\]
Now, if $d_{m}^{0}\neq d_{k-1}^{\infty}$, the deseired result
follows from Proposition \ref{ls}.
\end{proof}

\begin{remark}\label{rek2}
If $A_{0}=A_{\infty}$, then the decompositions of $X$ in Lemma
\ref{sl2} and in \eqref{red} are related to the same eigenvalue
problem \eqref{se3}. For the case proved in this section, namely
Theorem \ref{st2} (ii), we have $m>k$ and $V_{0}\oplus V_{+}\subset
X^{-}$.

As in \cite{MR2488059}, it is then easy to show that $\varphi$ has a
local linking with respect to the decomposition $X^{-}=\left(
V_{-}\cap X^{-}\right)  \oplus\left(  V_{0}\oplus V_{+}\right) $.
Then the local linking version of the three critical points theorem
\cite{MR1127041,MR802575} yields the desired result, we don't need
Theorem \ref{t1}.

On the other hand, if $A_{0}\neq A_{\infty}$, then the above
inclusion of the decompositions is false, the local linking property
of $\Phi$ does not descend to $\varphi$. Hence the local linking version of the three critical points theorem is not applicable, and our Theorem \ref{t1} is
crucial.
\end{remark}

\paragraph*{acknowledgements}
This work was completed while S.B. Liu was visiting the Institute of
Mathematics, Chinese Academy of Sciences; and the Institute of Mathematics, Peking University. S.B. Liu would like
to thank both institutes for invitation and hospitality.

% Note about the bibliography list
% If you don't want the hyperlinks attaching to
% the titles of references wark,
% just delete the '%' below:
% \renewcommand{\href}[2]{#2}

\end{document}